\documentclass[12pt, preprint]{amsart}
\usepackage{amsmath, amssymb, amscd}%
\usepackage[backref]{hyperref}
\usepackage[alphabetic, backrefs, lite]{amsrefs}
\usepackage{fullpage, url}
\usepackage{lscape}
\usepackage{color}%

    \newtheorem{theorem}{Theorem}[section]
    
    \newtheorem{proposition}[theorem]{Proposition}

    \newtheorem{remark}[theorem]{Remark}

\newcommand{\C}{{\mathbb C}}
\newcommand{\F}{{\mathbb F}}
\newcommand{\Q}{{\mathbb Q}}
\newcommand{\R}{{\mathbb R}}
\newcommand{\Z}{{\mathbb Z}}

\newcommand{\BBp}{{\mathbb B}_{p, \infty}}
\newcommand{\WW}{{\mathbb W}}

\newcommand{\Fbar}{{\overline{\F}}}
\newcommand{\phibar}{{\overline{\phi}}}
\newcommand{\Phibar}{{\overline{\Phi}}}
\newcommand{\Dtilde}{{\widetilde{D}}}
\newcommand{\Ktilde}{{\widetilde{K}}}
\newcommand{\Ftilde}{{\widetilde{F}}}
\newcommand{\Phitilde}{{\widetilde{\Phi}}}

\newcommand{\pp}{{\mathfrak p}}

\newcommand{\calC}{{\mathcal C}}
\newcommand{\calD}{{\mathcal D}}

\newcommand{\calM}{{\mathcal M}}

\newcommand{\calO}{{\mathcal O}}

\newcommand{\calR}{{\mathcal R}}

\newcommand{\calT}{{\mathcal T}}
\newcommand{\calU}{{\mathcal U}}

\newcommand{\calY}{{\mathcal Y}}
\newcommand{\calZ}{{\mathcal Z}}
\newcommand{\OO}{{\mathcal O}}
\newcommand{\Star}{\textasteriskcentered}

\DeclareMathOperator{\im}{Im}

\DeclareMathOperator{\divv}{div}
\DeclareMathOperator{\ord}{ord}

\DeclareMathOperator{\Jac}{Jac}

\DeclareMathOperator{\M}{M}

\DeclareMathOperator{\Sp}{Sp}
\DeclareMathOperator{\SL}{SL}

\DeclareMathOperator{\Tr}{Tr}
\DeclareMathOperator{\Disc}{Disc}
\DeclareMathOperator{\Norm}{N}

\DeclareMathOperator{\Sym}{Sym}
\DeclareMathOperator{\End}{End}
\DeclareMathOperator{\Hom}{Hom}
\DeclareMathOperator{\Aut}{Aut}

\begin{document}

	\title[Igusa class polynomials]{Igusa class polynomials, 
		embeddings of quartic CM fields, and arithmetic intersection theory}
	\author[Grundman, Johnson-Leung, Lauter, Salerno, Viray, Wittenborn]{
				Helen Grundman, Jennifer Johnson-Leung,
			 	Kristin Lauter, Adriana Salerno, Bianca Viray, Erika Wittenborn}
	
	\begin{abstract}
		Bruinier and Yang conjectured a formula for an intersection number on the arithmetic Hilbert modular surface, $\calC\calM(K).T_m$, where $\calC\calM(K)$ is the zero-cycle of points corresponding to abelian surfaces with CM by a primitive quartic CM field K, and $T_m$ is the Hirzebruch-Zagier divisors parameterizing products of elliptic curves with an $m$-isogeny between them. In this paper, we examine fields not covered by Yang's proof of the conjecture. We give numerical evidence to support the conjecture and point to some interesting anomalies. We compare the conjecture to both the denominators of Igusa class polynomials and the number of solutions to the embedding problem stated by Goren and Lauter.
	\end{abstract}
	
	\maketitle

	\section{Introduction}\label{sec:intro}
		
		Hilbert class polynomials are polynomials whose zeros are $j$-invariants of elliptic curves with complex multiplication (CM).  Similarly, Igusa class polynomials are polynomials whose zeros are Igusa invariants of genus 2 curves whose Jacobians have CM by a quartic CM field $K$.  Igusa class polynomials can be hard to compute, mostly because recovering the coefficients from approximations requires a bound on the denominators. A precise description of the denominators would greatly improve the estimated running time of current algorithms to compute Igusa class polynomials~\cite{Spallek, vW, Weng, EL, GHKRW}.		

		The first attempt to understand the primes appearing in the factorization of the denominators was given in~\cite{Lauter}, where it was conjectured that the primes are bounded by the absolute value of the discriminant $d_K$ of the field $K$, and that the primes divide $d_K - x^2$, for some integer $x$.  The divisibility condition actually follows from work of Goren~\cite{Gor97} along with the observation that primes in the denominators are primes of superspecial reduction for the abelian surfaces with CM by $K$. 
In~\cite{GL}, Goren and Lauter proved a bound on the primes which appear in the factorization of the denominators.  In recent work~\cite{GL10}, they have given a bound on the denominators (i.e. bounding the powers of the primes which appear); however, this bound is not sharp.
		
		Igusa invariants can be expressed as rational functions in Eisenstein series and cusp forms.  Using this description it can be shown that the primes dividing the denominators are (up to cancellation) the primes $p$ such that there exists a genus $2$ curve $C$ with CM by $K$ and a prime ideal $\pp | p$ with
		\[
			\Jac(C) \cong E \times E' \pmod \pp.
		\]
		Goren and Lauter obtained the bound on such primes by showing that these primes are exactly the primes for which there is a solution to the {\it embedding problem}: the problem of giving an embedding $\calO_K \hookrightarrow \End(E \times E') \subset M_2(\BBp)$, where $E$, $E'$ are supersingular elliptic curves over $\Fbar_p$ and $\BBp$ is the quaternion algebra over $\Q$ ramified at only $p$ and $\infty$.  Such decomposable Jacobians can be viewed as points on the Siegel moduli space where the arithmetic intersection of the CM($K)$-cycle with the divisor of the cusp form $\chi_{10}$ is nonzero at $p$.  
		
		This arithmetic intersection number on the Siegel moduli space can be related to a sum of intersection numbers of Hirzebruch-Zagier divisors and CM($K$)-cycles on the Hilbert modular surface
associated to the real quadratic subfield of $K$.  
In~\cite{BY}, Bruinier and Yang conjectured a formula for the intersection number $\mathcal{CM}(K).\calT_m$,  where $\mathcal{CM}(K)$ parameterizes abelian surfaces with complex multiplication by a primitive quartic CM field $K$, and $\calT_m$ is a Hirzebruch-Zagier divisor parameterizing products of elliptic curves with an $m$-isogeny between them when $m$ is split in the real quadratic field.  The results of Bruinier and Yang can be viewed as a generalization of the celebrated formula of Gross and Zagier~\cite{GZ-SingularModuli}. 
Bruinier and Yang conjectured their formula for these arithmetic intersection numbers under the assumption that the real quadratic subfield of $K$ has prime discriminant.  Their formula relates the intersection numbers to the number of ideals of a certain norm in the reflex field of $K$.  In~\cite{Yang1, Yang2}, Yang gives a detailed treatment of the embedding problem and uses it, along with other techniques, to prove the conjectured intersection formula under certain conditions on the CM field $K$. 
	
		For a given CM field $K$, the conjectural formula of Bruinier and Yang is straightforward to compute.  Thus, if the Bruinier-Yang formula held for all $K$, the formula would give a way to compute the denominators precisely (up to cancellation).

		In this paper, we compare the number of solutions to the embedding problem, the denominators in the Igusa class polynomials, and the conjectural formula of Bruinier and Yang for the 13 CM fields found in~\cite{vW}. Only 6 of these fields are covered by Yang's proof of the conjecture.  We give numerical evidence to support the conjecture and point to some interesting anomalies and some possible explanations for these discrepancies.  
Such a project was suggested in the Introduction and Section 9 of~\cite{BY}, and we give here the first numerical evidence comparing the conjecture with denominators of Igusa class polynomials and counting solutions to the embedding problem.  This comparison required developing an algorithm to count the number of solutions to the embedding problem for a given primitive quartic CM field $K$ and prime number $p$.  The algorithm is given in Section~\ref{algorithm} and the \texttt{Magma} code for computing the values is given in the Appendix.  
	
		The paper is organized as follows.	Section~\ref{sec:CMAbelSurf} contains background on abelian surfaces with complex multiplication by a primitive quartic CM field $K$.  Section~\ref{sec:igusa_inv} contains a description of Igusa invariants and Igusa class polynomials.  It also gives a brief introduction to the algorithms for computing Igusa class polynomials and their connection with generating genus 2 hyperelliptic curves for cryptography. 
Section~\ref{sec:intersection-theory} introduces arithmetic intersection theory on the arithmetic Hilbert moduli space and explains how it relates to the denominators of Igusa class polynomials.  In this section, we also describe the Bruinier-Yang conjecture and the cases in which it is known.  Section~\ref{sec:emb} explains how the denominators of Igusa class polynomials relate to the embedding problem.  Section~\ref{sec:Examples} gives an algorithm to count the number of solutions to the embedding problem. Section~\ref{sec:numerical-data} presents our numerical data and analysis of the findings.
	
		This paper is concerned with computing these intersection numbers from three different points of view: denominators of Igusa class polynomials, arithmetic intersection numbers via the conjectured Bruinier-Yang formula, and embeddings of a CM number field into a matrix algebra over a quaternion algebra.  Thus the paper necessarily requires the introduction of a number of objects from different points of view, and the background is spread over many different references and papers.  Therefore, for the readers' convenience, we have tried to include here as much background as possible.  Most sections are self-contained, but Section~\ref{sec:intersection-theory} leaves some of the background material to the references.	
		
	\section{Abelian surfaces with CM by $\OO_K$}\label{sec:CMAbelSurf}

		A {\it CM field} is a totally imaginary quadratic extension of a totally real number field of finite degree.  For a CM field $K$, let $K^+$ denote the real quadratic subfield of $K$.  If $\{\phi_1, \dots, \phi_{2n} \}$ are the $2n$ complex embeddings of $K$, then a {\it CM-type} $\Phi$ for $K$ is a choice of $n$ distinct complex embeddings, no two of which are complex conjugates of one another. When $K^+ = \Q$, CM fields are imaginary quadratic fields. 
	
		An {\it abelian variety with CM by $K$} is a pair $(A, \iota)$ where $A$ is an abelian variety and $\iota$ is an isomorphism of $K$ into $\End^0(A) := \End(A)\otimes \Q$.  We say $(A, \iota)$ has type $(K,\Phi)$ if $A$ has CM by $K$ and the analytic representation of $\End^0(A)$ is equivalent to the direct sum of the $n$ isomorphisms $\phi_i \in \Phi$.  One can define a CM-type to be {\it primitive} if all of the abelian varieties over $\C$ of that type are {\it simple}, i.e.  have no abelian subvarieties other than $\{0\}$ and itself.  A CM-type $(K,\Phi)$ is primitive if and only if $K=K^+(\xi)$, where $\xi$ is an element of $K$ such that $\xi^2$ is totally negative and $\im(\xi^{\phi_i}) > 0$ for $i=1, \dots, n$~\cite[p. 61]{Shimura}.  For quartic CM fields, the only case where non-primitive CM-types arise is the case when $K$ is biquadratic~\cite[p. 64]{Shimura}.  We will restrict ourselves to the primitive (i.e. non-biquadratic) case in this article.
	
		Let $(K,\Phi)$ denote a quartic CM field $K$ with CM-type $\Phi$. Define the reflex field of $(K,\Phi)$ to be
		\[
			K_\Phi = \Q\left(\sum_{\phi\in\Phi} x^\phi \mid x \in K\right).
		\]
		If $K$ is Galois cyclic, then we have $K=K_\Phi$ (\cite[\S 8.4]{Shimura}).
		
		Shimura gives a construction of a principally polarized complex abelian variety of a given CM type in~\cite[\S 6.1, Theorem 2]{Shimura}: 
to any ideal $I \subseteq \OO_K$, associate the abelian variety $A_I = \C^2/\Phi(I)$ of dimension~2. It has endomorphism ring $\OO_K$ and we say it has {\it CM by $\OO_K$}. All abelian varieties with CM by $\OO_K$ are obtained via this construction.  A {\it principal polarization} of $A_I$ is an isomorphism between $A_I$ and its {\it dual\/} $\widehat{A_I} = \C^2/\Phi(\overline{I}^{-1} \calD_K^{-1}),$ where 
		\[
			\calD_K^{-1} = \{x \in K \mid \Tr_{K/\Q}(x\OO_K) \subseteq \Z \}
		\]
		is the inverse different.  If $\pi$ is a totally imaginary element of $K$ which satisfies $\Phi(\pi) \in(i\R_{>0})^2$ and $\pi I \overline{I} = \calD_K^{-1}$, then the map $A_I \rightarrow \widehat A_I$ given by
		\[
			(z_1,z_2) \mapsto (\phi_1(\pi) z_1, \phi_2(\pi) z_2)
		\]
		is an isomorphism (\cite[p. 102--104]{Shimura}).

		Let $K$ be a fixed primitive quartic CM-field and fix a CM-type $\Phi: K \rightarrow \C^2$. Let $A/\C$ be a principally polarized abelian surface of CM-type $(K, \Phi)$ with endomorphism ring isomorphic to $\OO_K$. 
We define a group  $\mathfrak{C}(K) = \{ (\frak{a},\alpha) \}/\sim$ where $\frak{a}$ is a fractional ideal of $\OO_K$ such that $\frak{a}\overline{\frak{a}} = (\alpha)$, with $\alpha\in K^+$ totally positive.
Two pairs $(\frak{a},\alpha)$ and $(\frak{b},\beta)$ are equivalent if and only if there exists a unit $u\in K^*$ with $\frak{b} = u\frak{a}$ and $\beta = u\overline u \alpha$. The multiplication is defined componentwise, and $(\OO_K,1)$ is the identity element of $\mathfrak{C}(K)$.
		
		The group $\mathfrak{C}(K)$ naturally acts on the finite set $S(K,\Phi)$ of isomorphism classes of principally polarized abelian surfaces that have CM by $\OO_K$ of a given type $\Phi$. Every principally polarized abelian variety with CM by $K$ has an analytic representation $A_I$ determined by an ideal $I$ and a $\Phi$-positive element $\pi\in K$ giving the principal polarization. We now put
		\[
			(\frak{a},\alpha) \cdot (I,\pi) = (\frak{a} I, \alpha\pi)
		\]
		for $(\frak{a},\alpha) \in \mathfrak{C}(K)$. By~\cite[\S 14.6]{Shimura}, for a given $\Phi$ the action of $\mathfrak{C}(K)$ on $S(K,\Phi)$ is transitive and free. In particular, we have $|\mathfrak{C}(K)| = |S(K,\Phi)|$, and the sets are often identified.
		
		Abelian surfaces $A$ of type $(K, \Phi)$ defined by an ideal $\frak{a}$ have potentially zero, one, or two principal polarizations. 
There are zero principal polarizations only if there is no totally positive element $\alpha \in \OO_{K_0}$ with $\frak{a}\overline{\frak{a}} = (\alpha)$.
Otherwise there are $[U:U_1]$ principal polarizations on $A$, where $U$ is the group of all totally positive units of $K_0$, and $U_1$ is the subgroup of norms of units from $K$. 
		
		For quartic primitive CM fields $K$, there are two different possible CM types up to complex conjugation.  Let $\rho$ be complex conjugation, and let $\phi$ be a non-trivial embedding of $K$ into $\C$ other than complex conjugation.  A priori, the 4 possible CM types are: $\Phi = (1, \phi)$, $\Phibar = (\rho,\rho\phi)$, $\Phi'=(1, \phibar)$, $\Phibar' = (\rho,\rho\phibar)$. We denote by $K_\Phi$ (resp. $K_{\Phi'}$, etc.) the set of principally polarized abelian varieties of CM type $\Phi$ (resp. $\Phi'$, etc.).
		
		The description above identifies isomorphism classes of principally polarized abelian varieties of a given type $\Phi$ with the set $\{(\frak{a},\alpha)\}/\sim$, where $\frak{a}$ is an ideal of $\OO_K$ such that $\frak{a}\overline{\frak{a}} = (\alpha)$, $\alpha$ a totally positive element of $\OO_{K_0}$.  The associated abelian variety is $A=\C^2/\Phi(\frak{a})$ and the embedding of $\OO_K$ into $\End(A)$ sends $\beta \in \OO_K$ to the diagonal matrix with $(\beta^{\phi_1}, \beta^{\phi_2})$ on the diagonal acting by multiplication on $\C^2$.
		
		We can show that $K_\Phi = K_\Phibar$ by sending $(\frak{a},\alpha)$ to $(\overline{\frak{a}},\alpha)$. For the same reason, $K_{\Phi'} =  K_{\Phibar'}$.  If $K$ is cyclic, then $\phi^2$ is complex conjugation and sending $(\frak{a},\alpha)$ to $(\frak{a}^\phi,\alpha^\phi)$ shows that $K_{\Phi} =  K_{\Phi'}$.

	\section{Genus 2 curves with CM}\label{sec:igusa_inv}
	
		In cryptography, it is useful to be able to construct genus 2 curves over a finite field $\F_q$ with a given number of points on its Jacobian, since they can be used to implement discrete-log cryptosystems. Many algorithms to generate such curves have been developed using complex multiplication, all of which require computing certain invariants first studied by Igusa~\cite{IgusaGenus2}. The closely related Igusa class polynomials have rational coefficients. The known algorithms for computing Igusa class polynomials all require a good bound on the primes appearing in the denominators of these invariants.  One of the main goals of this project is to better understand exactly which primes appear in those denominators. In this section, we will review the definitions of the Igusa invariants and explain several algorithms for computing Igusa class polynomials.
		
		\subsection{Generating genus 2 curves for cryptography}
		\label{sec:crypto}
	
			To construct a genus 2 curve $C$ over a finite field $\F_q$ with a given number of points $N$ on its Jacobian, one method is to compute a related quartic CM field $K$ such that a curve $C$ with $\#J(C)(\F_q)=N$ has CM by $K$.

			For an ordinary genus 2 curve $C$ over a prime field $\F_q$, let $N_1 = \#{C}(\F_q)$ and $N_2= \#{C}(\F_{q^2})$.  Then 
			\begin{equation} \label{Jacpoints}
				\#J({C})(\F_q) = (N_1^2 + N_2)/2 - q.
			\end{equation}
			To find a curve $C$ over $\mathbb{F}_q$ such that $\#J(C)=N$, first find $N_1$ and $N_2$ in the Hasse-Weil intervals for $\mathbb{F}_q$ and $\mathbb{F}_{q^2}$ satisfying relation (\ref{Jacpoints}).  Next, set 
			\[
				s_1 := q + 1 - N_1
			\] 
			and 
			\[
				s_2 := \frac{1}{2}\left(s_1^2 + N_2 - 1 - q^2\right).
			\]
			Then the quartic polynomial $f(t) = t^4 - s_1t^3 + s_2t^2 - qs_1t + q^2$ is the Weil polynomial of a genus 2 curve as long as the exceptional cases listed in~\cite[Theorem 1.2]{HNR} are avoided.  Note that if $s_2$ is prime to $q$, then the Jacobian is ordinary~\cite[p.2366]{Howe}.  Under those conditions and assuming that $f(t)$ is an irreducible polynomial in $\Q[t]$, the Jacobian of the curve has endomorphism ring equal to an order in the quartic CM field $K= \Q[t]/(f(t))$.  Also, if $K$ can be written in the form $K=\Q(i\sqrt{a+ b\sqrt{d}})$, with $a,b,d \in \Z$ and $d$ and $(a,b)$ square-free, then $K$ is a primitive CM field if and only if $a^2-b^2d$ is not a square~\cite[p.135]{KW}.
		
			To find a curve with CM by $K$, one method is to compute its Igusa invariants by evaluating certain Siegel modular functions at CM points associated to $K$ in the Siegel moduli space.  Igusa invariants will be explained in the next section.
		
		\subsection{Igusa invariants and Igusa class polynomials}
	
			Recall that the $j$-invariant of an elliptic curve can be calculated in two ways. One can compute it as a value of the modular $j$-function on a lattice defining the elliptic curve as a complex torus over $\mathbb{C}$, or directly from the coefficients of the equation defining the elliptic curve. Similarly, the three Igusa invariants of a genus 2 curve can be calculated in two different ways.
			
			Let $y^2 = f(x)$ be a hyperelliptic curve, where $f(x)$ is a sextic with roots $\alpha_1,\dots,\alpha_6$ and leading coefficient $a_6$. Igusa defined invariants $A,B,C,D$ as certain symmetric functions of the roots~\cite{IgusaGenus2}. For example, $D=a_6^{10}\prod_{i<j}(\alpha_i-\alpha_j)^2$ is the discriminant. The ring of rational functions of the coarse moduli space for hyperelliptic curves of genus 2 is generated by the absolute Igusa invariants, which can be defined as:
			\begin{equation}\label{eqn:igusa-disc}
			 	i_1 := \frac{A^5}{D}, \quad i_2 := \frac{A^3B}{D},
				\quad i_3 := \frac{A^2C}{D}.
			\end{equation}
			This choice of generators is not unique. These invariants are a generalization of the $j$-invariant for elliptic curves in the sense that, if $i_1$ is non-zero and the characteristic is not $2$ or $3$, these invariants agree for two smooth genus $2$ curves if and only if the two curves are isomorphic over an algebraically closed field. Note that if $f(x)$ defines a smooth genus $2$ curve, then $i_1, i_2, i_3$ are well defined.
	
			The invariants can also be defined in terms of modular functions on the Siegel moduli space.  In~\cite[p. 195]{Igusa62}, Igusa defined normalized Siegel modular cusp forms $\chi_{10}$ and $\chi_{12}$ of weights $10$ and $12$ in terms of the Siegel-Eisenstein series $E_w$ for $w = 4,6,10,12$.  There he used the term {\em normalized} to mean that the leading coefficient of the Fourier expansion is $1$.  For our purposes, we will want to work with modular forms that have integral Fourier coefficients which are relatively prime. 
By working with $4\chi_{10}$ and $12\chi_{12}$, we obtain modular forms normalized in this second sense.

			In~\cite[Section 5.2]{GL}, it was noted that the absolute Igusa invariants defined above, $i_1$, $i_2$, $i_3$, can also be expressed in terms of these Eisenstein series and cusp forms as follows:
			\[
				i_1 = 2\cdot3^5\frac{\chi_{12}^5}{\chi_{10}^{6}}, \quad
				i_2 = 2^{-3}\cdot3^3\frac{E_4\chi_{12}^3}{\chi_{10}^4}, \quad
				i_3 = 2^{-5}\cdot3\left( \frac{E_6\chi_{12}^2}{\chi_{10}^{3}} 
					+ 2^2\cdot3\frac{E_4\chi_{12}^3}{\chi_{10}^4} \right),
			\]
			(where the parentheses in the expression for $i_3$ were mistakenly left out in~\cite{GL}).  It was also remarked there that it follows from the formulas given by Igusa in~\cite[p. 848]{Igusa1} that these  $i_1$, $i_2$, $i_3$ coincide with the invariants defined in terms of theta functions and used for computation by~\cite{vW, Weng, Lauter}.  However, it is worth noting that in order to work with invariants which are quotients of normalized modular forms in the second sense, one should actually work with $2^{-3}i_1$, $2i_2$, $2^3i_3$ as Spallek did~\cite{Spallek}.

			{\it Igusa class polynomials} are the genus $2$ analogue of Hilbert class polynomials.  Just as we need three Igusa invariants to define an isomorphism class instead of one $j$-invariant, we have a triple of Igusa class polynomials for a CM field $K$ instead of one Hilbert class polynomial.  This triple of polynomials is defined as follows.
			\begin{equation} \label{Igusapoly}
			 	H_{\ell}(X):=\prod_{\tau}(X-i_{\ell}(\tau)), \quad \ell = 1,2,3,
			\end{equation}
			\noindent where the $i_{\ell}$ are the absolute Igusa invariants defined above, and the product is taken over all $\tau \in \Sp(4,\Z)\backslash\mathfrak{h}_2$ such that the associated principally polarized abelian variety has CM by $\OO_K$.  Igusa class polynomials have rational coefficients~\cite[Satz 5.8]{Spallek} (as opposed to integral coefficients as in the case of Hilbert class polynomials).  In the next section, we describe how this plays a key role in the difficulty of computing these polynomials.

		\subsection{Algorithms for computing Igusa class polynomials}
		\label{subsec:generating-curves}
			
			The complex analytic method for computing genus 2 curves with a given number of points on the Jacobian is analogous to the Atkin-Morain algorithm for elliptic curves~\cite{AM}. Recall that the Atkin-Morain algorithm computes the Hilbert class polynomial of an imaginary quadratic field $K$ by evaluating the modular $j$-invariants of all elliptic curves with CM by $K$ to high precision.  The complex analytic method~\cite{Spallek, vW, Weng} for generating genus 2 curves takes as input a primitive quartic CM field and evaluates the Igusa invariants to high precision to form the class polynomials with coefficients in $\Q$.  Computing the Igusa invariants to high precision is difficult to do efficiently.  To give an algorithm requires a bound on the denominators of the class polynomials, and no such bound was known at the time of those papers.

			Given a triple of roots of these polynomials, Mestre gave an algorithm to recover the equation of a genus 2 curve from its invariants~\cite{Mestre}.  So to find a genus 2 curve, first find roots of the polynomials modulo a prime number $p$ and then use Mestre's algorithm to construct a genus 2 curve over $\mathbb{F}_p$ whose Jacobian has CM by $K$. 
			
			In~\cite{EL}, Lauter and Eisentr\"{a}ger present a {\em CRT algorithm}, an alternative to the complex analytic algorithm to compute the Igusa class polynomials. The CRT algorithm takes as input a quartic CM field $K$ \emph{and} a constant $c$ such that all denominators divide $c$.  The algorithm first computes the Igusa class polynomials modulo small primes $\ell$ by finding all curves with the right zeta function and the right endomorphism ring modulo $\ell$. It then forms the three polynomials whose roots are the invariants of those curves. Finally, it uses the Chinese Remainder Theorem (CRT) and the constant $c$ to obtain the Igusa class polynomials with coefficients in $\Q$.

			There is another alternative approach to computing Igusa class polynomials using $p$-adic arithmetic (\cite{GHKRW}).  The method is similar to the complex analytic algorithm, except the values are calculated to a high $p$-adic precision, instead of a complex precision.  This algorithm also requires a bound on the denominators.

			Notice that all three algorithms require a bound on the denominators, and their running time is directly affected by the sharpness of the bound. In the complex analytic and $p$-adic algorithms, it is required to bound the amount of precision needed in the approximation of the $j_i$-values.  In the CRT algorithm, a common multiple of the denominators is needed.  Thus, an exact formula for the denominators would be very useful. 

	\section{Arithmetic intersection theory and denominators of Igusa class 
	polynomials}\label{sec:intersection-theory}
		
		In this section we describe how the denominators of Igusa class polynomials are related to arithmetic intersection numbers on the Siegel moduli space of principally polarized abelian surfaces and explain how this relates to intersection numbers on the Hilbert modular surface and the conjectural formula of Bruinier and Yang for the intersection number $\calC\calM(K)$.$\calT_m$. 

		It is a classical fact that the zero locus of $\chi_{10}$ on the coarse moduli space of abelian surfaces consists of exactly those abelian surfaces that decompose as a product of elliptic curves with the product polarization.  The arithmetic analogue of this statement was proved in~\cite[Cor 5.1.2]{GL}, that if a prime $\frak{p}$ divides the denominator of $(f/(4\chi_{10})^k)(\tau)$, for $\tau$ a CM point corresponding to a smooth curve $C$ and $f$ a Siegel modular form of weight $10k$ with integral Fourier coefficients with gcd $1$, then $C$ has bad reduction modulo $\frak{p}$.

		The absolute Igusa invariants are defined above as rational functions of the Siegel modular forms $\chi_{10}$, $\chi_{12}$,  $\psi_4$, and $\psi_6$, with only powers of $\chi_{10}$ in the denominators. Although this definition of the invariants is given in terms of genus $2$ curves over $\C$, the Igusa invariants can also be defined for genus $2$ curves over finite fields (as in formula~\ref{eqn:igusa-disc} above).  Furthermore, the Igusa invariants evaluated at CM points are algebraic numbers and can be reduced modulo suitable primes.  Avoiding bad primes in the denominators, we have~\cite[Thm 2]{EL} 
		\[
			i_{\ell}(C \bmod \pp) = i_{\ell}(C) \bmod \pp.
		\]
		Thus the valuation of the denominator of the constant term of $H_1$ at a prime $p$ is the sum of the order of the poles of $2\cdot3^5\frac{\chi_{12}^5}{\chi_{10}^{6}}$ at the CM-points over $\Fbar_p$.  We have an analogous statement for the Igusa invariants $H_2$ and $H_3$.  Since $\chi_{12}, \psi_4$, and $\psi_6$ have no poles, the orders of the poles of $i_{\ell}$ are bounded above by suitable multiples of the order of zeros of $\chi_{10}$.  Experimentally, we observe that usually $\chi_{12},  \psi_4$, and $\psi_6$ do not share zeros of high multiplicity with $\chi_{10}$, so the order of the poles should be close to the order of zeros of a suitable power of $\chi_{10}$. Computing the order of zeros of $\chi_{10}$ is equivalent to computing the arithmetic intersection number  $\divv(\chi_{10}).\calC\calM(K)$ (explained below) of the divisor of $\chi_{10}$ with the cycle of CM points associated to $K$.

		To compute the arithmetic intersection number $\divv(\chi_{10}).\calC\calM(K)$, first note that, up to a power of 2, $\chi_{10} = \Psi_{1, \textup{Siegel}}^2$  (\cite[Remark 9.3]{BY}), where $\Psi_{1, \textup{Siegel}}$ is the normalized Siegel-Borcherds product of discriminant $1$ whose divisor is $G_1$, the Humbert surface of discriminant $1$.  Thus at all primes except possibly $2$, we need to compute $2G_1.\calC\calM(K)$. The Humbert surface $G_1$ can be pulled back to a sum of Hirzebruch-Zagier divisors $T_m$ on the Hilbert moduli space, and so the quantity to be computed can be expressed as a sum of 
intersection numbers of the form $\calC\calM(K)$.$\calT_m$ on the Hilbert moduli space. 

		We now explain all of this in more detail.  The presentation in the rest of this section follows closely the treatment in~\cite{Yang1, Yang2, BBK}.  For further background, good references are the books of~\cite{vdG, Goren, KRY}.
In Section~\ref{subsec:Hilbert}, we describe the Hilbert moduli space and the moduli interpretation of Hirzebruch-Zagier divisors and CM-cycles. In Section~\ref{intersection}, we explain arithmetic intersection theory and the meaning of  $\calC\calM(K)$.$\calT_m$.  In Section~\ref{subsec:relation}, we explain the	relationship between the Humbert surface and Hirzebruch-Zagier divisors.  In Section~\ref{subsec:BY}, we describe the Bruinier-Yang conjectural formula for these intersection numbers.

		\subsection{The Hilbert moduli space and Hirzebruch-Zagier divisors}
		\label{subsec:Hilbert}
			Fix $F$, a real quadratic extension of $\Q$ with different $\calD_F$.  The Hilbert moduli space parametrizes principally polarized abelian surfaces $A$ with {\it real multiplication by $\OO_F$}, i.e. with a fixed embedding $\iota\colon \OO_F \hookrightarrow \End(A)$.  More precisely, the arithmetic Hilbert moduli space $\calM$ is the moduli stack that associates a scheme $S$ over $\Z$ to the set of triples $(A, \iota, \lambda)$ where
				\begin{enumerate}
					\item $A$ is an abelian 
							surface over $S$, 
					\item $\iota \colon \OO_F \hookrightarrow \End_S(A)$ gives 
						the structure of real multiplication, and
					\item $\lambda \colon \calD_F^{-1} \rightarrow
					 	\Hom_{\OO_F}(A, A^\vee)^{\Sym}$ is a homomorphism of 
						$\OO_F$-modules that takes totally positive elements to 
						$\OO_F$-linear symmetric polarizations and satisfies the 
						Deligne-Pappas condition that $\calD_F^{-1} \otimes A 
						\rightarrow A^\vee$ is an isomorphism of abelian 
						schemes.
				\end{enumerate}
				The $\C$-points of $\calM$ can be identified with the usual Hilbert moduli space $M=\SL_2(\OO_F) \backslash \mathfrak{H}^2$ described in the article of Grundman in this volume~\cite{Grundman}.

			Now suppose that $m$ is a prime number that splits in $F$.  Let $Y_0(m) = \Gamma_0(m) \backslash \mathfrak{H}$ be the modular curve parameterizing pairs of elliptic curves with an $m$-isogeny.  Let $\calY_0(m)$ be the moduli stack over $\Z[1/m]$ of pairs of elliptic curves with an $m$-isogeny.  Bruinier, Burgos Gil, and K\"uhn~\cite[Prop 5.11, 5.13]{BBK} construct a map from $\calY_0(m)$ into $\calM$.  They prove that the image of the minimal compactification of  $\calY_0(m)$ is the Zariski closure of the Hirzebruch-Zagier divisor, $\calT_m$. 

			Let $K$ be a primitive quartic CM number field with $K^+ = F$.  Let $\calC\calM(K)$ be the moduli stack over $\Z$ that represents the moduli problem assigning a base scheme $S$ to the set of triples $(A, \iota, \lambda)$, where $\iota\colon \OO_K \rightarrow \End_S(A)$ gives the $\OO_K$-action on $A$ such that the Rosati involution defined by $\lambda$ induces complex conjugation on $\OO_K$, and $(A, \iota|_{\OO_F}, \lambda) \in \calM(S)$. Bruinier and Yang also refer to the direct image of $\calC\calM(K)$ in $\calM$ by $\calC\calM(K)$.

		\subsection{Arithmetic Intersection Theory} \label{intersection}
			Following~\cite[Chapter 2]{KRY} and~\cite{Yang1, Yang2}, we define the arithmetic intersection number of two cycles on a Deligne-Mumford stack of dimension $n$ that is proper and flat over $\Z$.  Let $\calZ_1$ and $\calZ_2$ be two cycles that intersect properly, i.e. their codimensions are $r_1$ and $r_2$, with $r_1+r_2 = n$, and $\calZ_1 \cap \calZ_2 = \calZ_1 \times_\calM \calZ_2$ is a Deligne-Mumford stack of dimension $0$.  Then the arithmetic intersection number (written logarithmically) is
		\begin{align*}
			\calZ_1.\calZ_2 & = \sum_{p} 
				\sum_{x \in \calZ_1 \cap \calZ_2(\Fbar_p)}
				\frac{1}{\# \Aut(x)} \log \# 
				\widetilde{\OO}_{\calZ_1 \cap \calZ_2, x}\\
				& = \sum_{p} \sum_{x \in \calZ_1 \cap \calZ_2(\Fbar_p)}
				\frac{1}{\# \Aut(x)} i_p(\calZ_1, \calZ_2, x ) \log p,
		\end{align*}
		where $\widetilde{\OO}_{\calZ_1 \cap \calZ_2, x}$ is the local ring of  $\calZ_1 \cap \calZ_2$ at $x$, and $i_p(\calZ_1, \calZ_2, x ) = \mbox{Length } \tilde{\OO}_{\calZ_1 \cap \calZ_2, x}$ is the local intersection at $x$.
		
			We are interested in the arithmetic intersection of the cycles ${\calC\calM}(K)$ and $\calT(m)$ on the moduli stack $\calM$.  Using the moduli interpretation of the cycles $\calC\calM(K)$ and $\calT_m$, it is clear that a geometric point $x$ in $\calC\calM(K)\cap \calT_m = \calC\calM(K) \times_\calM \calT_m$ is indexed by a pair $(\phi\colon E \rightarrow E', \iota)$, where $\phi \in \calY_0(m)(\overline{\F_p})$ and $\iota\colon \OO_K \rightarrow \End_{\OO_F}(B)$ is an embedding such that the Rosati involution associated to $\lambda$ induces complex conjugation on the image of $K$ and where $(B, \lambda)$ is the polarized abelian surface in $\calT_m$ which is the image of $\phi$. Denote by $i_p(\phi, \iota )=i_p(\calC\calM(K), \calT(m),x)$ the local intersection number at a point $x = (\phi, \iota)$.  It follows from deformation theory that $i_p(\phi, \iota ) = \mbox{Length} \left(\WW[[t,t']]/I\right)$,  where $\WW$ is the Witt ring of $\Fbar_p$ and $I$ is the minimal ideal such that $\phi$ can be lifted to an isogeny modulo $I$ and the embedding $\iota$ can be lifted to an embedding modulo $I$~\cite[Eq. (3.7)]{Yang2}.

		\subsection{The relationship between the Hilbert and Siegel moduli 
			spaces}\label{subsec:relation}
			Let $M$ be the Hilbert modular surface associated to the real quadratic field of discriminant $D$, $F=\mathbb{Q}(\sqrt{D})$. Recall that forgetting real multiplication gives a morphism $\phi_D$ from the Hilbert moduli space $M$ to the Siegel moduli space $X=\Sp(4,\mathbb{Z})\backslash\mathfrak{h}_2$. The image $G_D\subset X$ of $M$ is known as the Humbert surface of discriminant $D$.  For every $G_D$, there is a normalized integral Siegel Borcherds product $\Psi_{D,Siegel}$ whose divisor is equal to $G_D$ (up to multiplication by $\pm1$).  Let $\mathcal{G}_D$ denote the flat closure of $G_D$ in the Siegel moduli stack representing principally polarized abelian surfaces.

			The pullback to the Hilbert modular surface $M$ of a Humbert surface $G_{D'}$ under the morphism $\phi_D$ is a linear combination of Hirzebruch-Zagier divisors~\cite[p.215]{vdG}:
				\begin{equation} \label{pullback}
					\phi_D^*( G_{D'} ) = \sum_{\substack{ x \in \Z_{\geq0}\\ 
						x < \sqrt{DD'}\\ x^2 \equiv DD'(4)}} T_{(DD' - x^2)/4},
				\end{equation}
				for $D'$ a positive discriminant, and the same is true of the flat closures of these cycles in $\calM$. By Theorem 2.8 in~\cite{BY} we also know that the pullback to $M$ of the Borcherds product $\Psi_{D',\textup{Siegel}}$ is a Hilbert-Borcherds product with divisor (\ref{pullback}).  In our case, the denominators of the Igusa functions are powers of $\chi_{10}$, and $\chi_{10}=\Psi_{1,\textup{Siegel}}^2$ up to a power of $2$ and its divisor is $2G_1$.   So we can now express the quantity that we want to compute, $2\mathcal{G}_1.\calC\calM(K)$ as a sum 
				\[
					2\sum_{\substack{ x >0 \\ \frac{D-x^2}{4} \in \Z_{\ge 0}}}
					 \calC\calM(K).\calT_{(D - x^2)/4}.
				\]
			
		\subsection{Bruinier-Yang Conjecture}\label{subsec:BY}
			Let $K=F(\sqrt{\Delta})$ be a non-biquadratic quartic CM field $K$, let $\Phi$ be a CM type, and let $(\Ktilde,\Phitilde)$ be a fixed reflex field and type. Bruinier and Yang assume that $F=\Q(\sqrt{D})$ is a real quadratic field with $D$ a prime congruent to $1$ modulo $4$.	The real quadratic subfield of $\Ktilde$ is $\Ftilde=\Q(\sqrt{\Delta \Delta'})$, where $\Delta'$ is the real conjugate of $\Delta$. Let $\widetilde{D}$ be the norm of the relative discriminant of $K/F$ and let $d_{\widetilde{K}/\widetilde{F}}$ be the relative discriminant of ${\widetilde{K}/\widetilde{F}}$. Bruinier and Yang's conjectural formula for the arithmetic intersection number on the Hilbert moduli space is:
			\[
				\calC\calM(K).\calT_m = \frac{1}{2}\sum_{p} b_m(p) \log p,
			\]
			where 
			\[
				b_m(p) \log p = \sum_{\pp|p} 
						\sum_{t = \frac{n + m \sqrt{\Dtilde}}{2D} 
						\in d_{\Ktilde/\Ftilde}^{-1}, |n| < m\sqrt{\Dtilde} } 
						B_t(\pp).
			\]
			The ideals $\frak{p}|p$ are ideals of $\OO_{\Ftilde}$ and 
			\[
				B_t(\frak{p}) = \begin{cases}
									0 & \mbox{ if }\frak{p} \mbox{ is split in }
									\Ktilde\\ (\ord_{\pp}t + 1)
									\rho(t d_{\Ktilde/\Ftilde}\frak{p}^{-1} )
									\log |\frak{p}| & \mbox{ if }\frak{p}
									\mbox{ is not split in }\Ktilde
								\end{cases}
			\]
			where $\rho(\frak{a})= \#\{\frak{A} \subset \OO_{\Ktilde} \mid N_{\Ktilde/\Ftilde} \frak{A} = \frak{a}  \}$.  
					
			Yang has proved the conjecture~\cite{Yang1, Yang2} under the assumptions that $D \equiv 1 \mod{4}$ is prime, $\Delta\Delta' \equiv 1 \mod{4}$ is prime, and $\OO_K$ is a free $\OO_F$-module generated by $1$ and $\frac{1}{2} \left(w + \sqrt{\Delta}\right)$, for some $w \in \OO_{\Ftilde}$.  Note that Yang sets $\Dtilde=\Delta\Delta'$, which is equal to the norm of the relative discriminant of $K/F$ under his assumptions.  They are not necessarily equal in general, in particular, they differ for many of the fields listed in~\cite{vW}.  In our computations, we use the convention from~\cite{BY}, that $\Dtilde = \Norm_{F/\Q}(\Disc_{K/F})$ and we find that with this convention, the predicted arithmetic intersection number agrees with the computed values for the number of embeddings and the denominators of Igusa class polynomials in many cases.

	\section{Quaternion algebras and embeddings of CM fields}\label{sec:emb}
		In this section, we explain the connection given in ~\cite{GL} between denominators of Igusa class polynomials and certain embeddings of the ring of integers of $K$ into the matrix algebra $\M_2(\BBp)$, where $p$ is a prime number and $\BBp$ is the unique (up to isomorphism) quaternion algebra over $\Q$ which is ramified at only $p$ and $\infty$.
		
		Let $C$ be a smooth projective genus 2 curve such that $\Jac(C)$ has CM by $\calO_K$.  There is a number field $L$  such that all of the endomorphisms of $\Jac(C)$ are defined over $L$ and thus $\Jac(C)$ has everywhere good reduction.  It is observed in~\cite[\S 4.1]{GL} that if $K$ is a primitive CM field and $C$ has bad reduction at a rational prime $p$, then for a prime $\pp|p$ of $L$, the reduction of $\Jac(C)$ modulo $\pp$ is isomorphic over $\calO_L/\pp$ to a product of supersingular elliptic curves $E_1\times E_2$ with the product polarization. Moreover, we have that $E_1$ is isogenous to $E_2$ as otherwise 
		\[
			\End(E_1\times E_2) = \End(E_1) \times \End(E_2)
		\]
		contradicting $K\hookrightarrow\End^0(\Jac(C))$ with $K$ primitive.  Goren and Lauter prove that in fact the existence of a smooth projective genus 2 curve $C$ that has CM by $\OO_K$ and has bad reduction at $\pp$ is equivalent to the existence of an embedding 
		\[
			\iota:\calO_K \hookrightarrow \End(E_1\times E_2)
		\]
		such that the Rosati involution coming from the product polarization induces complex conjugation on $\OO_K$~\cite[Lemma 4.1.1 and Thm. 4.2.1]{GL}.

		\subsection{Relating $\End(E_1 \times E_2)$ to a subring of $M_2(\BBp)$}

			We view an element $f \in \End(E_1 \times E_2)$ as a matrix 
$\begin{pmatrix} f_{1,1} & f_{1,2} \\ f_{2,1} & f_{2,2}\end{pmatrix}$ 
where $f_{i,j}\in \Hom(E_j, E_i)$.  Then the composition of endomorphisms $f \circ g$ corresponds to multiplication of the matrices and the Rosati involution $f\mapsto f^{\vee}$ corresponds to
			\[
				\begin{pmatrix} 
					f_{1,1} & f_{1,2} \\ 
					f_{2,1} & f_{2,2}
				\end{pmatrix}
				\mapsto
				\begin{pmatrix} 
					f_{1,1}^{\vee} & f_{2,1}^{\vee} \\ 
					f_{1,2}^{\vee} & f_{2,2}^{\vee} 
				\end{pmatrix}
			\]

			Fix a supersingular elliptic curve $E/\Fbar_p$ and an isomorphism $\psi\colon \End(E) \to \calO$ where $\calO$ is a maximal order in $\BBp$.  Note that by using $\psi$ we can view elements of $\End(E)\otimes\Q$ as elements of $\BBp$.  Given an element $\phi \in \Hom(E, E')$, we obtain an embedding $\Hom(E', E) \to \End(E)$ by mapping $f \mapsto f\circ \phi$.  Thus we can view $\Hom(E', E)$ as an ideal of $\End(E)$ or, by using the isomorphism $\psi$, as a left ideal $I$ of $\calO$.  Note that if we replace $\phi$ by a different endomorphism $\phi'$, this corresponds to right multiplying $I$ by $\psi(\phi^{-1}\phi') \in \BBp$ (where $\phi^{-1}\phi' = \left(\phi^{\vee}\phi'\right)/\deg(\phi)$ can be viewed as an element of $\End(E)\otimes \Q$).  Thus a different choice of $\phi$ results in a different ideal, but one in the same \emph{ideal class}.  In fact, the map
			\begin{eqnarray*}
				\left\{\mbox{supersingular curves over }\Fbar_p\right\} & \to &
				\left\{\mbox{left ideal classes of }\calO\right\}\\
					E' & \mapsto & 
					\psi(\Hom(E', E)\phi) \mbox{ for some }\phi \in \Hom(E, E')
			\end{eqnarray*}
			is a bijection.  Given a morphism $\phi$ and the ideal $I =  \psi(\Hom(E', E)\phi)$, we can also identify the group $\Hom(E, E')$ with the ideal $I^{-1}$ by the map
			\[
			\Hom(E, E') \to \BBp, \quad g \mapsto \psi\left(\phi^{-1}g\right).
			\]
	
			We would like to view $\End(E')$ as a subring of $\BBp$ in a compatible way.  Let $\phi$ denote a fixed element of $\Hom(E, E')$. Then we obtain a map
			\[
			\End(E') \to \BBp, \quad g \mapsto \psi\left(\phi^{\vee}g\phi\right)/\left(\deg(\phi)\right).
			\]
			Since this map is an embedding, the image $\calO'$ is isomorphic to $\End(E')$ and therefore is a maximal order in $\BBp$.  One can check that $\calO'$ is the right order of $I := \psi(\Hom(E', E)\phi)$.  Thus, given a fixed $\phi\in\Hom(E,E')$ and a fixed isomorphism $\psi\colon \End(E)\to \calO\subseteq\BBp$, we can identify $\End(E \times E')$ with the subring of $M_2(\BBp)$,
			\[
				\begin{pmatrix} \calO & I \\ I^{-1} & \calO' \end{pmatrix},
			\]
		where $I$ and $\calO'$ are defined as above.  Under this identification the Rosati involution can be described as
			\[
				\begin{pmatrix}
					x & y \\
					z & w
				\end{pmatrix}^{\vee} =
				\begin{pmatrix}
					x^{\vee} & \Norm(I)z^{\vee} \\
					y^{\vee}/\Norm(I) & w^{\vee}
				\end{pmatrix},
			\]
		where $N(I) := \gcd\{\gamma\gamma^{\vee} \mid \forall \gamma \in I\}$ is the norm of $I$.  From this point forward, we will refer interchangeably to elements of $\End(E\times E')$ and elements of $\calR_{(E, E', \phi)} := \begin{pmatrix} \calO & I \\ I^{-1} & \calO' \end{pmatrix}$.

		\subsection{Isomorphisms of principally polarized CM abelian 
		varieties}\label{subsec:CMiso}
			Two principally polarized abelian varieties with complex multiplication $(A, \iota_1, \lambda_1)$, $(A, \iota_2, \lambda_2)$ are isomorphic either if $\iota_1$ and $\iota_2$ differ by an automorphism of $\OO_K$ or if there exists $\sigma \in \Aut(A)$ such that $\sigma^{-1}\iota_1\sigma = \iota_2$ and $\lambda_1 = \sigma^\vee \lambda_2 \sigma$.  
			We would like to describe the set of $\sigma \in \Aut(E \times E')$ such that conjugation by $\sigma$ preserves the polarization.  Preserving the polarization is equivalent to 
				\[
					\left(\sigma^{-1}\psi\sigma\right)^{\vee} = 
					\sigma^{-1}\psi^{\vee}\sigma
					\quad \mbox{ for all } \psi \in \End(E \times E').
				\]
			Thus $\sigma\sigma^{\vee} = \pm 1$.  Let $\begin{pmatrix} p & q\\ r& s \end{pmatrix} \in \calR_{(E, E', \phi)}$ be the element associated to $\sigma$.  The previous condition implies that $q = r = 0$ and both $p$ and $s$ are isomorphisms, or $p=s=0$ and $q$ and $r$ are isomorphisms.  Note that the latter possibility only occurs if $E$ and $E'$ are isomorphic.  Let $\calU_{(E, E', \phi)}$ denote the set of elements in $ \calR_{(E, E', \phi)}$ of this form, i.e. satisfying the condition that $\begin{pmatrix} p & q\\ r& s \end{pmatrix}\begin{pmatrix} p & q\\ r& s \end{pmatrix}^{\vee} = \pm 1$.

	\section{Counting the number of embeddings} \label{sec:Examples}
		Let $\calO_K$ denote the ring of integers of $K$.  If the real quadratic subfield $F$ of $K$ has class number $1$, then there exists an $\eta \in \calO_K$ such that $\calO_K = \calO_F[\eta]$.  Let $D$ denote the discriminant of $F$ and let $\omega := \frac{1}{2}\left(D + \sqrt{D}\right)$, a generator of $\calO_F$.  Denote $\Tr(\eta) = \alpha_0 + \alpha_1\omega$ and $\Norm(\eta) = \beta_0 + \beta_1\omega$.

		\begin{proposition}\label{prop:matrix-desc}
			Let $E, E'$ be supersingular curves over $\Fbar_p$.  Fix an isomorphism $\psi\colon \End(E) \to \calO$ where $\calO$ is a maximal order and an element $\phi \in \Hom(E, E')$.
			Determining an embedding 
			\[
				\iota \colon \calO_K \hookrightarrow \End(E \times E')
			\]
			is equivalent to giving two elements $\Lambda_1, \Lambda_2 \in \calR_{(E, E', \phi)}$ such that 
			\begin{align*}\label{eqn:lambdas}
				\Lambda_1\Lambda_2  = &\Lambda_2 \Lambda_1 \\
				\Lambda_1^2 - D\Lambda_1 + \frac{D^2 - D}{4}  = &0\\
				\tag{*} \Lambda_2 + \Lambda_2^{\vee}  = & 
					\alpha_0 + \alpha_1\Lambda_1 \\
				\Lambda_2\Lambda_2^{\vee}  = & 
						\beta_0 + \beta_1\Lambda_1.
			\end{align*}
		\end{proposition}
		
		\begin{proof}
			Let $\iota$ be an embedding.  Then $\Lambda_1= \iota(\omega)$ and $\Lambda_2 = \iota(\eta)$ satisfy the conditions above.  In the other direction, it is clear that the map that sends $\omega\mapsto \Lambda_1$ and $\eta \mapsto \Lambda_2$ gives a well-defined homomorphism from $\calO_K$ to $\calR_{(E, E', \phi)}$.  Therefore it remains to show that the map is injective.  Assume that the map is not injective  and let $J$ denote the kernel.  If $J$ is nontrivial then $\OO_K/J$ embeds into $\calR_{(E, E', \phi)}$.  However, since $\OO_K/J$ is a finite abelian group and $\BBp$ is a torsion-free $\Z$-module, this is impossible.
		\end{proof}
		
		We note that the idea behind this lemma is similar to the ideas in~\cite{GL} which used the embedding problem to obtain a bound on the primes dividing the denominator.  Goren and Lauter studied pairs of matrices $M_1$, $M_2$ in $\calR_{(E, E', \phi)}$ where $M_1$ satisfies the minimal polynomial of $\sqrt{D}$ and $M_2$ satisfies the minimal polynomial of $\sqrt{\Delta}$ where $K = F(\sqrt{\Delta})$.
		
		In light of Proposition~\ref{prop:matrix-desc} and Section~\ref{subsec:CMiso}, the number of pairs of $(E \times E', \iota)$, counted up to isomorphism, is equal to
		\[
			\sum_{\substack{\mbox{\tiny unordered pairs}\\E, E' }} \# \frac{\left\{ (\Lambda_1, \Lambda_2) \in \calR_{(E, E', \phi)}^2	\mbox{ satisfying \ref{eqn:lambdas}}\right\}}{ (\Lambda_1, \Lambda_2) \sim (U\Lambda_1U^{-1}, U\Lambda_2U^{-1})} \quad \mbox{ where }\phi \in\Hom(E, E') \mbox{ and }U \in \calU_{(E, E', \phi)}
		\]
		We use this formula to count the number of embeddings.  Our algorithm is sketched below and our implementation in $\texttt{Magma}$ can be found in Appendix~\ref{app:code}.  We make no claims on the running time of the algorithm or the implementation; the algorithm was only implemented to obtain the numerical evidence in Table~\ref{table:data}.

		\subsection{Algorithm for counting solutions to the embedding problem}
		\label{algorithm}
			\begin{enumerate}
				\item (Initialize variables) Input $p, D, \alpha_0, \alpha_1, \beta_0, \beta_1, \Dtilde$.
				\item (Create representation of $\BBp$) Let $\BBp = (a, -p)$ where $a$ is the maximal negative integer such that $\BBp$ is ramified only at $p$ and $\infty$.
				\item (Create representations of $\calR_{(E, E', \phi)}$ for each unordered pair of elliptic curves $(E, E')$) We represent each $\calR_{(E, E', \phi)}$ by a triple $(B_1, B_2, B_3)$ where $B_1$ is a basis for $\calO$, $B_2$ is a basis for $I$ and $B_3$ is a basis for $\calO'$.  For each isomorphism class of supersingular elliptic curves, let $\calO$ be a maximal order isomorphic to the endomorphism ring.  Then for each left ideal class of $\calO$ let $I$ be an integral representative and let $\calO'$ be the right order.  Take bases to give representatives of $\calR_{(E, E', \phi)}$ for each unordered pair $(E, E')$.
			
				\begin{enumerate}
					\item (Find possibilities for $\Lambda_1$) The condition on $\Lambda_1$ implies that it is of the form
					\[
						\begin{pmatrix} s_{1,1} & s_{1,2} \\ 
							s_{1,2}^{\vee}/N(I) & D - s_{1,1} \end{pmatrix},
					\]
					where $s_{1,1}$ is an integer in $[(D - \sqrt{D})/2, (D + \sqrt{D})/2]$ and $s_{1,2}$ is an element in $I$ with norm equal to 
					\[
						-s_{1,1}^2 + Ds_{1,1} - (D^2 - D)/4.
					\]
					There are only finitely many possibilities for each of these, so we iterate over all of them.
				
					\begin{enumerate}
			
						\item (Find possibilities for $\Lambda_2$)  Denote the entries of $\Lambda_2$ by $t_{i,j}$.  The condition on $\Lambda_2 + \Lambda_2^{\vee}$ implies that $t_{2,1}$ is determined by $t_{1,2}$ and $\Lambda_1$ and that $t_{1,1}$ and $t_{2,2}$ have fixed trace.  The condition on $\Lambda_2\Lambda_2^{\vee}$ implies that $t_{1,1}, t_{1,2}, t_{2,2}$ have bounded norm.  This gives only finitely many possibilities for $t_{1,1}, t_{1,2}, t_{2,2}$ and we iterate through them all, checking whether the pair $(\Lambda_1, \Lambda_2)$ satisfies all of the conditions in (\Star).  If (\Star) is satisfied, then keep $(\Lambda_1, \Lambda_2)$ in the list, otherwise discard the pair.
			
						\item (Identify elements that differ by conjugation or automorphisms of $\OO_K$) There are only finitely many elements in $\calU_{(E, E', \phi)}$.  For a fixed pair $(\Lambda_1, \Lambda_2)$, iterate through all elements $U$ in $\calU_{(E, E', \phi)}$ and all $\sigma \in \Aut(\OO_K)$ and remove all pairs $(\sigma(U\Lambda_1U^{-1}), \sigma(U\Lambda_2U^{-1}))$.

					\end{enumerate}
				\end{enumerate}
				\item (Count the embeddings) Count the number of pairs $(\Lambda_1, \Lambda_2)$.
			\end{enumerate}

	\section{Numerical Data}\label{sec:numerical-data}

			In Table~\ref{table:data} we examine the $13$ fields studied by van Wamelen~\cite{vW}. These $13$ fields are the only fields $K$ such that there exists a genus $2$ curve $C$ \emph{over $\Q$} with CM by $K$, and they are all quartic Galois cyclic extensions of $\Q$.  For each field $K$, we list
			\begin{enumerate}
 				\item the factorization of the rational sixth root of the denominator of $\frac{\chi_{12}^5}{\chi_{10}^6}(\tau)$, where $\tau$ is a CM-point of $K$,
 				\item the conjectural value of $2\phi^*G_1.\calC\calM(K)$ as predicted by the Bruinier-Yang formula (this was computed via a \texttt{Magma} 
					program), and 
				\item the number of solutions to the embedding problem, computed using Algorithm~\ref{algorithm} and the \texttt{Magma} code listed in Appendix~\ref{app:code}.
			\end{enumerate}
			
			The three values listed above should agree if the number of solutions to the embedding problem is counted \emph{with multiplicity}.  Unfortunately, our representation of embeddings does not count multiplicity.  However, Yang's proof suggests that the value $\left(\ord_{\pp} t + 1\right)/2$ is the multiplicity of an embedding, under certain conditions.  We write the Bruinier-Yang values as iterated exponents, where the outer exponent denotes  $\left(\ord_{\pp} t + 1\right)/2$.  If there is no outer exponent, then this value is $1$.  This is the case in most examples, so it appears that most of the embeddings have no multiplicity.  We investigate the cases where the Bruinier-Yang conjecture does not agree with the numerical data for the other columns in Section~\ref{subsec:disc-reasons}.
			
			\begin{remark} 
				Since all of the fields in~\cite{vW} are Galois, there is only one CM-type up to isomorphism.  Therefore we multiply the Bruinier-Yang formula by $1/2$ since the CM-cycle in~\cite{BY} is CM($K, \Phi) + $CM$(K, \Phi')$ where $\Phi$ is not the complex conjugate of $\Phi'$.
			\end{remark}
			\begin{remark} 
				In the data we present in Table~\ref{table:data}, we set $\Dtilde =\Norm_{F/\Q}(\Disc_{K/F})$ as in the original conjecture~\cite{BY}.  In his proofs, Yang sets $\Dtilde=\Delta\Delta'$ which is equal to the norm of the relative discriminant of $K/F$ under his assumptions, but which is not equal in general.  
			\end{remark}

		\subsection{Analysis of data}\label{subsec:disc-reasons}
			The first remark to make is that, in every example in the table, the powers of the primes in the denominators of the Igusa invariants match the number of solutions to the embedding problem, except for the power of $2$ in four examples and the power of $3$ in the last row. The power of $3$ in the last row appears to be accounted for by multiplicity.  The power of $2$ is different in the four rows 2, 4, 9, and 11, and we explain why that is accounted for by cancellation in more detail in Section~\ref{subsec:powersof2}.

			The rest of this section is devoted to explaining the discrepancies when comparing the Bruinier-Yang conjecture with the denominators of the Igusa class polynomials.  Note that the six rows which are marked with a ($*$) are cases where the Bruinier-Yang conjecture is already proved by Yang, since in those cases both $D$ and $\Dtilde$ are primes congruent to $1 \pmod{4}$.  The two rows with a ($**$) in the table are not covered by the conjecture, since in those cases the real quadratic subfield has discriminant which is divisible by a power of $2$.  Next we remark that in all the cases which are covered by the Bruinier-Yang conjecture, that is, where $D \equiv 1 \pmod{4}$ is prime, the numerical data confirms the conjecture and shows that the 2 quantities computed for each field match, except for the power of $2$ in the examples already mentioned above and the power of $23$ in row $11$.  We cannot explain the discrepancy in the power of $23$ in row $11$.
			\begin{landscape}
			\begin{table}[h]\label{table:data}
				\begin{tabular}{|c|c|c|c|c|}
					\hline
					$\Dtilde$ & CM field &  denominators of Igusa 
						& Bruinier-Yang formula& Number of embeddings\\
						& & invariants & (written multiplicatively) & 
						(written multiplicatively) \\
					\hline
					$5$ & $\Q\left(\zeta_5\right)^*$ 
						& $1$ 
						& $1$ 
						& $1$ \\
					\hline
					$2^5$ & $\Q\left(\sqrt{-2 + \sqrt{2}}\right)^{**}$ 
						& $1$ 
						& $(2^4)^{-1/2} (2^8)^{-3/2}$ 
						& $2$ \\
					\hline
					$13$ & $\Q\left(\sqrt{-13 + 2\sqrt{13}}\right)^*$ 
						& $1$ 
						& $1$ 
						& $1$\\
					\hline
					$2^6 5$ & $\Q\left(\sqrt{-5 + \sqrt{5}}\right)$ 
						& $1, \; 11^2$ 
						& $(2^2)^{3/2} 11^2$ 
						& $2^2 \cdot 11^2$\\
					\hline
					$5\cdot13^2$ & $\Q\left(\sqrt{-65 + 26\sqrt{5}}\right)$ 
						& $11^2, \; 31^2 41^2$ 
						& $11^2 31^2 41^2 $ 
						& $11^2 31^2 41^2$\\	
					\hline
					$29$ & $\Q\left(\sqrt{-29 + 2\sqrt{29}}\right)^*$ 
						& $5^2$ 
						& $5^2$ 
						& $5^2$\\
					\hline
					$5\cdot 17^2$ & $\Q\left(\sqrt{-85 + 34\sqrt{5}}\right)$ 
						& $71^2, \; 11^2 41^2 61^2$ 
						& $11^2 41^2 61^2 71^2$
						& $11^2 41^2 61^2 71^2$\\
					\hline 
					$37$ & $\Q\left(\sqrt{-37 + 6\sqrt{37}}\right)^*$ 
						& $3^2 11^2$ 
						& $3^2 11^2$ 
						& $3^2 11^2$\\
					\hline
					$2^5 5^2$ & $\Q\left(\sqrt{-10 + 5\sqrt{2}}\right)^{**}$ 
						& $7^2 23^2, \; 7^2 17^2 23^2$
						& $(2^{28})^{-3/2} (2^{2})^{-1} (2^{12})^{-1/2}$
						
						& $2^2 7^4 17^2 23^4$\\
						
						& & & $	\times (5^{32})^{1/2} 5^{14}7^6 $ &\\
						& &	& $	\times 17^4 23^4 31^2 41^2 79^2$ &\\
					\hline
					$5^2 13$ & $\Q\left(\sqrt{-65 + 10\sqrt{13}}\right)$ 
						& $3^2, \; 3^2 53^2$
						& $3^4 53^2$ 
						& $3^4 53^2$\\
					\hline
					$2^6 13$ & $\Q\left(\sqrt{-13 + 3\sqrt{13}}\right)$ 
						& $3^2 23^2,\; 3^2 23^2 131^2$ 
						& $(2^2)^{\frac{3}{2}} 3^4 23^2 131^2$ 
						& $2^2 3^4 23^4 131^2$\\
					\hline
					$53$ & $\Q\left(\sqrt{-53 + 2\sqrt{53}}\right)^*$ 
						& $17^2 29^2$ 
						& $17^2 29^2$
						& $17^2 29^2$\\
					\hline
					$61$ & $\Q\left(\sqrt{-61 + 6\sqrt{61}}\right)^*$ 
						& $3^4 5^2 41^2$ 
						& $\left(3^2\right)^2 5^2 41^2$ 
						& $3^2 5^2 41^2$ \\
					\hline
				\end{tabular}
				\caption{ Numerical data: precise descriptions of the values 
					computed can be found in Section~\ref{sec:numerical-data}.
					$\qquad\qquad\qquad\qquad\qquad\qquad$
					\Star~:~Conjecture has been proved in this case since 
					$\Dtilde \equiv 1 \pmod{4}$ is a prime 
					$\qquad\qquad\qquad\qquad\qquad\qquad$
					\Star\Star~:~Conjecture (as stated) does not include this 
					case;~\cite{BY} assumed that the real quadratic subfield has 
					prime discriminant}
			\end{table}
		\end{landscape}
				
		Other than the power of $23$, the discrepancies in the table occur in the four fields where $\Dtilde$ is highly divisible by $2$.  
In those cases, negative and/or fractional powers of $2$ appear in the Bruinier-Yang formula, plus other stray primes and powers in row nine.
To explain that, observe that in~\cite{Yang1, Yang2}, Yang shows that the multiplicity is equal to 
			\[
				\frac{1}{2}\left(\ord_p\frac{m^2\Dtilde - n^2}{4D} + 1\right)
			\] 
			and proves under his assumptions that this value is equal to 
			\[
				\frac{1}{2}\left(\ord_{\pp} \frac{n + m \sqrt{\Dtilde}}{2D} 
				+ 1\right).  
			\]
			In particular, $\ord_p\frac{m^2\Dtilde - n^2}{4D}$ is odd.  However, these statements no longer hold for those 4 examples in the table, and this leads to the fractional exponents.   

			When $D$ is also even in addition to $\Dtilde$ being even,  Bruinier and Yang did not conjecture a formula, so the different values do not give a counterexample to the conjecture.  However, in the aim of obtaining a formula for all CM-fields, we suggest a possible explanation for the differences and a possible correction.  In Yang's papers, he proves that one should consider all pairs $m,n$ such that $\frac{m^2\Dtilde - n^2}{4D}$ is a positive integer divisible by $p$.  This proof relies on the fact that $D \equiv 1\pmod 4$.  Note that if one imposes the additional condition that $8m + n \equiv 0 \pmod{16}$, then one obtains intersection numbers $1$ and $7^4 17^2 23^4$ for the fields $\Q\left(\sqrt{-2 + \sqrt{2}}\right)$ and $\Q\left(\sqrt{-10 + 5\sqrt{2}}\right)$ respectively.  Away from $2$, these numbers agree with both the number of embeddings and the denominators.  This suggests that when the real quadratic subfield has discriminant congruent to $0$ modulo $4$, additional congruence conditions on the pairs $m,n$ are needed to correct the intersection formula.

		\subsection{Powers of $2$}\label{subsec:powersof2}
			In four examples in the table, rows 2, 4, 9, and 11, the power of $2$ in the denominator does not match the number of embeddings at the prime $2$.  The reason for that is cancellation of powers of $2$ in the numerator and denominator of Igusa invariants.  It was noted in~\cite[\S 6.2]{GL} that the phenomenon of cancellation could occur, especially for the prime $2$.  It was explained there that this is due to the fact that there are no smooth superspecial genus 2 curves in characteristic 2, despite the fact that superspecial primes are prevalent.  For these four fields, the prime $2$ is ramified, and so it follows from a generalization of the work of Goren~\cite{Gor97} given in~\cite[Section 3]{GL10} that $2$ is a superspecial prime, and thus these curves have bad reduction at $2$.  So indeed we do find solutions to the embedding problem at $2$ in those cases, and the only reason that $2$ does not appear in the denominator is because of cancellation with powers of $2$ in the numerator.
		
	\section*{Acknowledgements}
This work was started as our group project at the WIN workshop ``Women In Numbers'', held at the Banff International Research Station in November 2008.  We thank BIRS for hosting us and for the excellent working conditions there.  We also thank the Fields Institute, PIMS, NSA, Microsoft Research, and the University of Calgary for their generous sponsorship of the conference. We thank Christophe Ritzenthaler for many detailed comments to improve the paper and the third author thanks Eyal Goren and Tonghai Yang for many discussions.

	\appendix
	
	\section{\texttt{Magma} code for counting the number of embeddings}
	\label{app:code}
	
	\begin{verbatim}
		//Given a list of 4 rational numbers, q1, q2, q3, q4
		//return the minimal positive integer N such that
		//Nq1, Nq2 Nq3, Nq4 are integers
		function Denom(L)

		    denominators := [];
		    for i := 1 to 4 do //hardcoded the length
		        Append(~denominators, Denominator(L[i]));
		    end for;
		    return LCM(denominators);

		end function;

		/********************************************************************/
		//Check whether the element x is in the integral span of elements 
		//in the basis B
		function InIdeal(x, B)
		    A := Matrix([Coordinates(B[1]), Coordinates(B[2]),
		                 Coordinates(B[3]), Coordinates(B[4])]);
		    v := Vector(Coordinates(x));
		    s := Solution(A,v);

		    retval := true;
		    for n := 1 to 4 do
		        if not IsIntegral(s[n]) then
		            retval := false;
		        end if;
		    end for;
		    return retval;

		end function;

		/********************************************************************/
		//Return the Rosati involution of M.
		//Notice this depends on the Norm of the ideal I
		function Rosati(M,NI)

		    return Matrix([ [ Conjugate(M[1][1]), Conjugate(M[2][1])*NI],
		                    [ Conjugate(M[1][2])/NI, Conjugate(M[2][2])]]);
		end function;
		\end{verbatim}
		\pagebreak
		\begin{verbatim}
		/********************************************************************/
		//Find all elements in the integral span of the basis B with 
		//norm = Normx and trace = Tracex
		function FixedNormAndTrace(Normx, Tracex, B)

		    xvals := [];
		    Q := Parent(B[1]);
		    isq := -Norm(Q.1);
		    jsq := -Norm(Q.2);
		    A := Transpose( Matrix([ Coordinates(B[1]), Coordinates(B[2]),
		                     		 Coordinates(B[3]), Coordinates(B[4])])); 

		    denom1 := Denom(A[2]);
		    denom2 := Denom(A[3]);
		    denom3 := Denom(A[4]);
		    x0 := Tracex/2;
		    temp3 := Normx - x0^2;

		    if temp3 ge 0 then
		        for num3 := Ceiling(-Sqrt(Rationals()!(temp3*denom3^2/(isq*jsq)))) 
		            to Floor(Sqrt(temp3*denom3^2/(isq*jsq))) do

		            temp2 := temp3 - (num3/denom3)^2*isq*jsq;
		            for num2 := Ceiling(-Sqrt(temp2*denom2^2/(-jsq)))
		                to Floor(Sqrt(temp2*denom2^2/(-jsq))) do

		                temp1 := temp2 - (num2/denom2)^2*(-jsq);
		                for num1 := Ceiling(-Sqrt(temp1*denom1^2/(-isq)))
		                    to Floor(Sqrt(temp1*denom1^2/(-isq))) do

		                    x := x0 + num1/denom1*Q.1 + 
		                         num2/denom2*Q.2 + num3/denom3*Q.3;
		                    if (Norm(x) eq Normx) and (Trace(x) eq Tracex) 
		                        and InIdeal(x, B) then
		                        Append(~xvals, x);
		                    end if;
		                end for;
		            end for;
		        end for;
		    end if;

		    return xvals;
		end function;
		\end{verbatim}
		\pagebreak
		\begin{verbatim}
		/********************************************************************/
		function RemoveConjElts(solns, E, alpha0, alpha1, beta0, beta1, D)

		    n := 1;
		    while n le #solns do

		        // Loop over automorphisms of E
		        for r in FixedNormAndTrace(1,  0, E[1]) cat 
		                 FixedNormAndTrace(1,  1, E[1]) cat
		                 FixedNormAndTrace(1, -1, E[1]) cat
		                 FixedNormAndTrace(1,  2, E[1]) cat
		                 FixedNormAndTrace(1, -2, E[1]) do

		            // Loop over automorphisms of E'
		            for s in FixedNormAndTrace(1,  0, E[3]) cat 
		                     FixedNormAndTrace(1,  1, E[3]) cat
		                     FixedNormAndTrace(1, -1, E[3]) cat
		                     FixedNormAndTrace(1,  2, E[3]) cat
		                     FixedNormAndTrace(1, -2, E[3]) do

		                //Calculating NormI so that we can compute Rosati involution
		                I := LeftIdeal(QuaternionOrder(E[1]), E[2]);
		                NormI := Norm(I);

		                U := Matrix([[r,0],[0,s]]);
		                Uinv := Matrix([[Conjugate(r), 0],[0, Conjugate(s)]]);

		                //Remove embedding which is conjugate to given embedding by U
		                // First check that U does not fix embedding
		                if not (solns[n][1]*U eq U*solns[n][1] and 
		                        solns[n][2]*U eq U*solns[n][2]) then

		                    Exclude(~solns, [U*solns[n][1]*Uinv, 
		                        U*solns[n][2]*Uinv]);

		                end if;

		                //Remove embedding which is conjugate to the complex conjugate 
		                //of given embedding by U
		                // First check that U does not fix complex conjugate of 
		                // embedding                
		                if not (solns[n][1]*U eq U*solns[n][1] and 
		                        solns[n][2]*U eq U*Rosati(solns[n][2], NormI)) then

		                    size := #solns;
		                    Exclude(~solns, [U*solns[n][1]*Uinv, 
		                                     U*Rosati(solns[n][2], NormI)*Uinv]);
		                end if;

		                //if E and E' are isomorphic then remove embedding which are 
		                // conjugate and have the order of E, E' switched
		                if E[1] eq E[2] then

		                    U := Matrix([[0, r], [s, 0]]);
		                    Uinv := Matrix([[0, Conjugate(s)], [Conjugate(r), 0]]);

		                    if  U*Uinv ne Matrix([[1,0],[0,1]]) or 
		                        Uinv*U ne Matrix([[1,0],[0,1]]) then
		                        print "Error - Uinv incorrect!!", r, s;
		                    end if;

		                    if not (solns[n][1]*U eq U*solns[n][1] and 
		                            solns[n][2]*U eq U*solns[n][2]) then

		                        size := #solns;
		                        Exclude(~solns, [U*solns[n][1]*Uinv,
		                                         U*solns[n][2]*Uinv]);

		                    end if;
		                    if not (solns[n][1]*U eq U*solns[n][1] and 
		                            solns[n][2]*U eq U*Rosati(solns[n][2], NormI)) then
		                        size := #solns;
		                        Exclude(~solns, [U*solns[n][1]*Uinv, 
		                                         U*Rosati(solns[n][2], NormI)*Uinv]);
		                    end if;
		                end if;
		            end for;
		        end for;

		        n := n + 1;

		    end while;

		    return solns;

		end function;
		\end{verbatim}
		\pagebreak
		\begin{verbatim}
		/********************************************************************/
		// find Lambda1, Lambda2 satsifying embedding conditions and such that 
		// entries have fixed norm
		function FindSolns(E, alpha0, alpha1, beta0, beta1, D,
		                    Norms12, Normt12, Normt22, Tracet22, 
		                    Normt11, Tracet11, s11);

		    solns := [];


		    R1 := QuaternionOrder(E[1]);
		    R2 := QuaternionOrder(E[3]);
		    I := LeftIdeal(R1, E[2]);

		    // Create list of possible values in I with Norm = Norms12
		    s12traces := [-Floor(Sqrt(4*Norms12))..Floor(Sqrt(4*Norms12))];
		    s12vals := [];
		    for T in s12traces do
		        s12vals := s12vals cat FixedNormAndTrace(Norms12, T, E[2]);
		    end for;

		    for s12 in s12vals do

		        // loop over values in R1 with Norm = Normt11 and Trace = Tracet11
		        for t11 in FixedNormAndTrace(Normt11, Tracet11, E[1]) do

		            // loop over values in R2 with Norm = Normt22 and Trace = Tracet22
		            for t22 in FixedNormAndTrace(Normt22, Tracet22, E[3]) do

		                // Create list of possible values in I with Norm = Normt12
		                t12traces := [-Floor(Sqrt(4*Normt12))..Floor(Sqrt(4*Normt12))];
		                t12vals := [];
		                for T in t12traces do
		                    t12vals := t12vals cat FixedNormAndTrace(Normt12, T, E[2]);
		                end for;

		                for t12 in t12vals do

		                    t21 := 1/Norm(I)*(alpha1*Conjugate(s12) - Conjugate(t12));

		                    //define Lambda1 = omega and Lambda2 = eta
		                    omega := Matrix([[s11, s12], 
		                                     [Conjugate(s12)/Norm(I), D - s11]]);
		                    eta := Matrix([[t11, t12], [t21, t22]]);

		                    //check omega and eta live in correct subring 
		                    // and satisfy all of the embedding conditions
		                    if InIdeal(s12, E[2]) and InIdeal(t12, E[2]) and
		                        (omega*eta eq eta*omega) and 
		                        eta + Rosati(eta, Norm(I)) eq alpha0 + alpha1*omega and
		                        eta*Rosati(eta, Norm(I)) eq beta0 + beta1*omega then

		                        Append(~solns, [omega, eta]); //add to list of solutions
		                    end if;
		                end for;
		            end for;
		        end for;
		    end for;
		    return solns;

		end function;



		/********************************************************************/
		// Lists all possible pairs of matrices Lambda1, Lambda2 in R_{E, E'} 
		// such that
		// Lambda1*Lambda2 = Lambda2*Lambda1
		// Lambda1^2 - D*Lambda1 + 1/4*(D^2 - D) = 0
		// Lambda2 + Lambda2^{\vee} = alpha0 + alpha1*Lambda1
		// Lambda2*Lambda2^{\vee} = beta0 + beta1*Lambda1

		function ListAllEmb(alpha0, alpha1, beta0, beta1, D, p)

		    // Define constants associated to CM field
		    F := QuadraticField(D);
		    D := Discriminant(F);

		    trw := D;  //Trace(omega)
		    nw := Integers()!(1/4*(D^2 - D));  //Norm(omega)

		    Dtilde := 1/16*D^4*alpha1^4 + 1/2*D^3*alpha0*alpha1^3 - 1/8*D^3*alpha1^4 
		        - D^3*alpha1^2*beta1 + 3/2*D^2*alpha0^2*alpha1^2 
		        - 1/2*D^2*alpha0*alpha1^3 - 4*D^2*alpha0*alpha1*beta1 
		        + 1/16*D^2*alpha1^4 - 2*D^2*alpha1^2*beta0 + D^2*alpha1^2*beta1 
		        + 4*D^2*beta1^2 + 2*D*alpha0^3*alpha1 - 1/2*D*alpha0^2*alpha1^2 
		        - 4*D*alpha0^2*beta1 - 8*D*alpha0*alpha1*beta0 
		        + 4*D*alpha0*alpha1*beta1 - 2*D*alpha1^2*beta0 + 16*D*beta0*beta1 
		        - 4*D*beta1^2 + alpha0^4 - 8*alpha0^2*beta0 + 16*beta0^2;

		    // Give a presentation of B_{p, infinity}
		    if p ne 2 then
		        jsq := -p;
		        isq := -1;
		        while (IsSquare(FiniteField(p)!isq)) or 
		            (#RamifiedPrimes(QuaternionAlgebra<Rationals() | isq, jsq>) 
		            gt 1) do
		            isq := isq - 1;
		        end while;
		    else
		        jsq := -1;
		        isq := -1;
		    end if;

		    QQ := Rationals();
		    Q<i, j, k> := QuaternionAlgebra<QQ | isq, jsq>;
		    ksq := -isq*jsq;


		    // Orders = { O : O maximal order isomorphic to End(E) for E supersingular 
		    // curve over Fpbar}
		    // Note: if E and E' are not isomorphic, but have isomorphic endomorphism   
		    // rings, Orders contains two different maximal orders, one is identified 
		    // with End(E), the other with End(E')
		    OO := MaximalOrder(Q);
		    LI := LeftIdealClasses(OO);
		    Orders := [];
		    for I in LI do
		        Append(~Orders, RightOrder(I));
		    end for;

		    // Ends = list of triples, one triple for each (ordered) pair (E, E')
		    // Let O be a maximal order which is isomorphic to End(E), I a left ideal
		    // of O which can be identified with Hom(E', E), and O' the right order of I
		    // which is identified with End(E').  Then the triple associated to (E, E')
		    // is (Basis(O), Basis(I), Basis(O')).  The correspondence between 
		    // (O, I, O') and (E, E') is explained more in the paper.
		    Ends := [];
		    for O in Orders do
		        for I in LeftIdealClasses(O) do
		            Append(~Ends, [Basis(O), Basis(I), Basis(RightOrder(I))]);
		        end for;
		    end for;

		    // initialize list of solutions
		    solns := [];
		    number := 0;
		    //look for Lambda1, Lambda2 in R_{(E, E')} satisfying above conditions
		    // this is equivalent to searching for s11, t11 in O, s12, t12 in I,
		    // s22, t22 in O'.  Then
		    // Lambda1 = 
		    // [s11 s12]
		    // [s12^{\vee} s22]
		    // Lambda2 = 
		    // [t11 t12]
		    // [t12^{\vee} t22]
		    for E in Ends do

		        NormI := Norm(LeftIdeal(QuaternionOrder(E[1]), E[2]));

		        //possible values for s11.  s22 = D - s11
		        s11vals := [Ceiling( (D - Sqrt(D))/2 )..Floor( (D + Sqrt(D))/2)];

		        Esolns := [];

		        for s11 in s11vals do

		            a := s11; //remnants of old notation
		            delta := -a^2 + a*trw - nw;
		            m := delta;
		            Tracet11 := alpha0 + alpha1*s11;  //conditions from Lambda + Lambda*
		            Tracet22 := alpha0 + alpha1*(D - s11);

		            for n := Ceiling(-Sqrt(Max(m^2*Dtilde - 4*D, 0))) to 
		                Floor(Sqrt(Max(m^2*Dtilde - 4*D, 0))) do

		                //parts of old notation
		                Normx := 1/(2*D)*(n - (1/4*a^2*D^2*alpha1^2 
		                - 1/4*a*D^3*alpha1^2 + 1/16*D^4*alpha1^2 - 1/4*a^2*D*alpha1^2 
		                + 1/4*a*D^2*alpha1^2 - 1/8*D^3*alpha1^2
		                + a^2*D*alpha1*alpha0 - a*D^2*alpha1*alpha0 
		                + 1/4*D^3*alpha1*alpha0 + 1/16*D^2*alpha1^2 
		                - 1/4*D^2*alpha1*alpha0 + a^2*alpha0^2 - a*D*alpha0^2
		                + 1/4*D^2*alpha0^2 - 2*a^2*D*beta1 + 2*a*D^2*beta1 
		                - 1/2*D^3*beta1 - 1/4*D*alpha0^2 - 2*a*D*beta1 
		                + 1/2*D^2*beta1 - 4*a^2*beta0 + 4*a*D*beta0
		                - D^2*beta0 - D*beta0));

		                Normu := delta*(beta0 + beta1*a) - delta*Normx;

		                Tracexuc := beta1*delta - (D - 2*a)/delta*Normu;

		                Normv := delta^2*Normx + delta*(D - 2*a)*Tracexuc 
		                            + (D - 2*a)^2*Normu;

		                //conditions coming from fixed n
		                Normt11 := Normx;
		                Normt12 := Normu*NormI/delta;
		                Normt22 := Normv/delta^2;
		                Norms12 := NormI*delta;

		                //check n can give a solution which is integral
		                if  IsIntegral((m^2*Dtilde - n^2)/(4*D*p)) and
		                    IsIntegral(Normt11) and IsIntegral(Normt12) and
		                    IsIntegral(Normt22) then 

		                    Esolns := Esolns cat FindSolns(E,
		                            alpha0, alpha1, beta0, beta1, D,
		                            Norms12, Normt12, Normt22, Tracet22, 
		                            Normt11, Tracet11, s11);
									// finds all solutions with 
		                            // entries of Lambda2 having
		                            // fixed norm and fixed
		                            // trace (for some entries)
		                end if;
		            end for;
		        end for;

		        // remove elements which are conjugate
		        Esolns := RemoveConjElts(Esolns, E, alpha0, alpha1, beta0, beta1, D);  

		        if E[1] eq E[2] then
		            number := number + #Esolns;
		        else
		            number := number + #Esolns/2;
		        end if;
		        Append(~solns, Esolns);
		    end for;

		    return solns, Ends, Integers()!number;

		end function;
	\end{verbatim}


\end{document}